\documentclass[psamsfonts,reqno]{amsart}

\usepackage{amssymb,amsfonts}
\usepackage[all,arc]{xy}
\usepackage{enumerate}
\usepackage{mathrsfs}
\usepackage{ytableau}
\usepackage{url}
\usepackage[margin=3.5cm]{geometry}

\newtheorem{thm}{Theorem}[section]
\newtheorem{cor}[thm]{Corollary}

\newtheorem{lem}[thm]{Lemma}
\newtheorem{conj}[thm]{Conjecture}
\newtheorem{quest}[thm]{Question}

\theoremstyle{definition}
\newtheorem{defn}[thm]{Definition}

\newtheorem{notn}[thm]{Notation}

\theoremstyle{remark}
\newtheorem{rem}[thm]{Remark}

\makeatletter
\let\c@equation\c@thm
\makeatother
\numberwithin{equation}{section}

\title{The Benson - Symonds invariant for permutation modules}

\author{Aparna Upadhyay}
\address{Department of Mathematics\\ University at Buffalo, SUNY \\
244 Mathematics Building\\Buffalo, NY~14260, USA}

\email{aparnaup@buffalo.edu}

\date{23 May 2019}

\begin{document}

\begin{abstract}

In a recent paper, Dave Benson and Peter Symonds defined a new invariant $\gamma_G(M)$ for a finite dimensional module $M$ of a finite group $G$ which attempts to quantify how close a module is to being projective. In this paper, we determine this invariant for permutation modules of the symmetric group corresponding to two-part partitions using tools from representation theory and combinatorics.

\end{abstract}

\maketitle


\section{Introduction}

We determine the gamma invariant as defined by Dave Benson and Peter Symonds in \cite{B_S} for permutation modules of the symmetric group labelled by two-part partitions. We will call this invariant the Benson - Symonds invariant and it will be denoted by $\gamma$. \\

\begin{defn}\cite[Definition 1.1]{B_S} For a $\mathbf{k}G$-module $M$, we write $M=M'\oplus (proj)$ where $M'$ has no projective direct summands and $(proj)$ denotes a projective module. Then $M'$ is called the $core$ of $M$ and denoted $core_G(M)$. We write $c_n^G(M)$ for the dimension of $core_G(M^{\otimes n})$; it is well defined by the Krull-Schmidt Theorem. \\
\end{defn}

\begin{thm} \cite[Theorem 1.2 (i)]{B_S} If $M$ is a finite dimensional module of a finite group $G$ then $$\lim_{n \to \infty} \sqrt[n]{c_n^G(M)}$$ exists.
\end{thm}

Benson and Symonds define the invariant $\gamma_G(M)$ to be the above limit and they observe that this invariant satisfies some very interesting properties. It measures the non-projective proportion of $M^{\otimes n}$ in the limit. They compute the invariant for a few examples in \cite{B_S} using the computer algebra system \textbf{Magma}. In fact as  per the definition, this invariant is unlikely to be an integer in most cases. In this paper, we give a closed formula for the invariant of an infinite class of permutation modules. Note that this invariant is not yet known for any class of modules beyond what the authors themselves determined.
\smallskip

\begin{rem} \label{count} 
\begin{enumerate}[(i)]
\item \cite{B_S} The invariant $\gamma_G(M)$ is robust, in that $c_n^G(M)$ may be replaced by the number of composition factors of $core_G(M)$ or the number of composition factors of the socle of $core_G(M)$, and so on.
\item Moreover, if the total number of non-isomorphic indecomposable summands that ever occur in a decomposition of $core_G(M^{\otimes n})$ is finite, then $c_n^G(M)$ may be replaced by the number of indecomposable summands of $core_G(M^{\otimes n})$. This is easy to see by letting $\mathbf{a}$ denote the dimension of the non-projective summand with smallest dimension and $\mathbf{A}$ denote that of the largest. If $d_n^G(M)$ is the number of indecomposable summands of $core_G(M^{\otimes n})$, then $d_n^G(M)\cdot \mathbf{a}\leq c_n^G(M) \leq d_n^G(M)\cdot \mathbf{A}$. The remark follows by taking the limit of the $n$th root. Similarly, if a particular summand has the highest multiplicity for each $n$, then $c_n^G(M)$ can be replaced by the multiplicity of that particular summand. Let $\mathbf{n_0}$ be the total number of non-projective summands that ever occur and $e_n^G(M)$ denote the multiplicity of the summand with highest multiplicity, then $e_n^G(M) \leq d_n^G(M) \leq e_n^G(M) \cdot \mathbf{n_0}$ and the remark follows.
\end{enumerate}

\end{rem}
\smallskip

Our aim is to prove the following theorem:\\
\begin{thm} \label{mytheorem} Let $\lambda =(\lambda_1, \lambda_2)\vdash n$ and $M^\lambda$ be the permutation module of $\mathcal{S}_n$ over a field $\mathbf{k}$ of characteristic $p$, then
$$\gamma_{\mathcal{S}_n}(M^\lambda)={n-p \choose \lambda_1}+{n-p \choose \lambda_2}.$$
\end{thm}
In the next section, we briefly summarize the background results that will be used in the course of this paper. In section 3, we fix $P$ an elementary abelian $p$-subgroup of $\mathcal{S}_n$ of maximum possible rank and consider the restriction of $M$ to $P$. We then investigate the indecomposable summands of this restriction and also determine the multiplicities of each indecomposable summand. We will also understand the properties of these indecomposable summands and their tensor products with each other. In section 4, we will compute the Benson - Symonds invariant for the restriction of $M$ to $P$. Finally in section 5, we will show that the Benson - Symonds invariant for $M$ is in fact equal to the Benson - Symonds invariant for $M$ restricted to $P$. In section 6, we put together some interesting observations and directions for further investigation. In the process, we encountered and proved some very interesting and useful binomial identities which are compiled at the end of the paper in an appendix. \\

\section{Background}

This section aims to introduce some preliminary concepts needed to deduce our main results. 

\begin{thm}\label{maxElt} \cite[Theorem 7.2]{B_S} Let $M$ be a $\mathbf{k} G$-module. Then $$\gamma_G(M)=\max_{E\leq G} \gamma_E(M)$$ where the maximum is taken over the set of elementary abelian $p$-subgroups $E$ of $G$.
\end{thm}
\smallskip

\begin{thm} \cite[Lemma 2.10]{B_S} If $H$ is a subgroup of $G$ and $M$ is a $\mathbf{k}G$-module then $\gamma_H(M) \leq \gamma_G(M)$.
\end{thm}
\smallskip
Therefore, while using Theorem \ref{maxElt} it is enough to take the maximum over the maximal elementary abelian $p$-subgroups of $G$.\\

Recall the conjugacy classes of maximal elementary abelian $p$-subgroups in $\mathcal{S}_n$. Let $(\mathbb{Z}/p)^n \cong V_n(p) \hookrightarrow \mathcal{S}_{p^n}$ using notations from \cite[p.185]{EltAbBook}. We have the following theorem:

\begin{thm}\label{conjEltAb} \cite[p.185, Theorem 1.3]{EltAbBook} If $n=a_0 + i_1p + i_2 p^2 + ... + i_r p^r$ with $0 \leq a_0 < p$, $i_j \geq 0$ for $1 \leq j \leq r$, then there is a maximal elementary abelian $p$-subgroup of $\mathcal{S}_n$ corresponding to this decomposition 
$$\underbrace{V_1(p) \times ... \times V_1(p)}_{i_1} \times ... \times \underbrace{V_r(p) \times ... \times V_r(p)}_{i_r} $$
$$\subset \underbrace{\mathcal{S}_p \times ... \times \mathcal{S}_p}_{i_1} \times ... \times \underbrace{\mathcal{S}_{p^r} \times ... \times \mathcal{S}_{p^r}}_{i_r} \subset \mathcal{S}_n$$
and as we run over distinct decompositions $(i_1,...,i_r)$ these give the distinct conjugacy classes of maximal elementary abelian $p$-subgroups of $\mathcal{S}_n$.
\end{thm}

\smallskip

\begin{thm}\label{trivialPlus} \cite[Theorem 4.6]{B_S} If $p$ divides $|G|$ and $M$ is a $\mathbf{k} G$-module, then we have $\gamma_G(\mathbf{k}\oplus M)=1+\gamma_G(M)$.
\end{thm}

\smallskip

\section{Indecomposable summands of $M^{(n-r,r)}$ restricted to an elementary abelian $p$-subgroup}

\begin{notn} Throughout this paper, $M$ is the permutation module $M^{(n-r,r)}$ of the symmetric group $\mathcal{S}_n$ on the subgroup $\mathcal{S}_{\lambda}$, where $\lambda=(\lambda_1,\lambda_2)=(n-r,r)\vdash n$ over a field $\mathbf{k} $ of characteristic $p$. We let $n=kp+a_0$ for $0 \leq a_0 < p$ and $r=qp+b_0$ for $0 \leq b_0 < p$. We have a standard basis of $M$ consisting of ${n \choose r } $ tabloids and we will denote each basis element by the second row of the tabloid.
\end{notn}
In this section, we analyse the indecomposable summands of the permutation module $M$ when restricted to a maximal elementary abelian $p$-subgroup of $\mathcal{S}_n$ of rank $k$. We first consider the rank $k$ elementary abelian $p$-subgroup of $\mathcal{S}_n$ generated by $p$-cycles.
\begin{notn}
Let $P=<(1,2,...,p),(p+1,...,2p),...,((k-1)p+1,...,kp)>$. Let $M_P=M \downarrow _P$. Given a tabloid $t$, let $<t>$ denote the cyclic $\mathbf{k} P$-module generated by $t$. Let $B_i=\{(i-1)p+1,...,ip\}$, for $1\leq i \leq k$ and $B_{k+j}=\{kp+j\}$, for $1 \leq j \leq a_0$. Then, $\{B_1,B_2,...,B_{k+a_0}\}$ is a partition of the set $\{1,2,...,n\}$. We will call each $B_i$ a block. Note that $$|B_i| = 
\left\{\begin{array}{lr}
        p, & \text{when } 1 \leq i \leq k\\
        1, & \text{when } i>k
        \end{array} \right. $$ 
        
\end{notn}        
\smallskip
\textbf{(I) Multiplicities of indecomposable summands of} $M_P$\\
\\
Fix the standard basis of $M_P$ to be the ${n \choose r} $ tabloids of $M$. We prove each indecomposable summand of $M_P$ is cyclic and generated by an individual tabloid. Hence, the dimensions of the indecomposable summands are the same as the sizes of $P$-orbits in the set of tabloids. For this purpose, we will pick up each tabloid and consider the cyclic submodule that it generates.\\
\\
\textbf{Case 1:} $r=p$\\
\begin{enumerate}[(i)]
\item Observe that the tabloid $\overline{1,2,...,p}$ generates a 1-dimensional submodule of $M_P$. Similarly, there are $k$ copies of 1-dimensional modules denoted by $A_i ^{(0)}$, $1 \leq i \leq k$, where $A_i^{(0)}=<\overline{(i-1)p+1,...,ip}>$, each non-isomorphic as a $\mathbf{k}P$-module to the others (proved in Corollary \ref{iso}).\\

\item If $t$ is a tabloid whose last row has entries from the block $B_i$ and $B_j$ for $i\neq j $ and $i,j \leq k$, then dim ($<t>$) $\geq p^2$. Therefore, dim ($<t>$)$=p$ if and only if the entries of $t$ are from the block $B_i$ for a unique $i \leq k$ and at least one (or more) entry from any set $B_j$, $j > k$.\\
Now let us see how many such tabloids exist. We know that the last row of the tabloid $t$ has $p$ entries. Let $\mu =(\mu _1, \mu _2)$ be a composition of $p$ into 2 parts. If we choose $\mu _1$ entries from a fixed $B_i$, $ 1 \leq i \leq k$ and $\mu _2$ entries from $B_{k+1} \cup ... \cup B_{k+a_0}$, we obtain a tabloid which generates a $p$-dimensional submodule.\\
Fix $i \leq k$, then the total number of tabloids of the above kind possible are $$\sum_{\substack {\mu \models p \\ \text{into 2 parts}}} {p \choose \mu _1} {a_0 \choose \mu _2}.$$
Since each of these tabloids span a $p$-dimensional submodule of $M_P$, therefore the total number of $p$-dimensional submodules of $M_P$ that are generated by a single tabloid is $$\frac{1}{p}\sum_{\substack {\mu \models p \\ \text{into 2 parts}}} {p \choose \mu _1} {a_0 \choose \mu _2}.$$ 
We will show that these modules are isomorphic to each other in Corollary \ref{iso}.\\
The total number of non-isomorphic $p$-dimensional submodules with basis given by individual tabloids will depend on the choice of $i$, so that is equal to ${k \choose 1}=k$. Hence, the total number of $p$-dimensional submodules which may be generated by a single tabloid is equal to $$\frac{k}{p}\sum_{\substack {\mu \models p \\ \text{into 2 parts}}} {p \choose \mu _1} {a_0 \choose \mu _2}.$$ 
\\
\item A tabloid $t$ such that dim ($<t>$)$=p^d$, for any $d\geq 2$ can be obtained in two different ways. First, choose exactly $d$ distinct sets from the collection $\{B_1,B_2,...,B_k\}$. \\
\begin{enumerate}[(a)]
\item Consider $\mu$ a composition of $p$ into exactly $d$ parts and choose $\mu _i$ entries from a chosen $B_i$. Therefore, for a fixed $d$ element subset of the set $\{B_1,B_2,...,B_k\}$, the total number of tabloids of the above kind possible are $$\sum_{\substack {\mu \models p \\ \text{into \textit{d} parts}}}{p \choose \mu _1}{p \choose \mu _2}\cdots {p \choose \mu _d}.$$
Hence, for a fixed choice of $d$ element set we get $$\frac{1}{p^d}\sum_{\substack {\mu \models p \\ \text{into \textit{d} parts}}}{p \choose \mu _1}{p \choose \mu _2}\cdots {p \choose \mu _d} $$ $p^d$-dimensional submodules that are generated by a single tabloid.\\
\item Consider $\mu$ a composition of $p$ into exactly $d+1$ parts and choose $\mu _{d+1}$ entries from $B_{k+1} \cup ... \cup B_{k+a_0}$. Therefore, for a fixed $d$ element subset of the set $\{B_1,B_2,...,B_k\}$, the total number of tabloids of the above kind possible are $$\sum_{\substack {\mu \models p \\ \text{into \textit{d}+1 parts}}}{p \choose \mu _1}{p \choose \mu _2}\cdots {p \choose \mu _d}{a_0 \choose \mu _{d+1}}$$
and we get $$\frac{1}{p^d}\sum_{\substack {\mu \models p \\ \text{into \textit{d}+1 parts}}}{p \choose \mu _1}{p \choose \mu _2}\cdots {p \choose \mu _d} {a_0 \choose \mu _{d+1}} $$ $p^d$-dimensional submodules that are generated by a single tabloid.\\
\end{enumerate}
For our fixed choice of $d$ element subset of $\{B_1,B_2,...,B_k\}$, the submodules obtained are isomorphic. The number of non-isomorphic submodules depends on the choice of the $d$ element subset of $\{B_1,B_2,...,B_k\}$, which can be chosen in ${k \choose d}$ ways. Hence, the total number of $p^d$-dimensional submodules which may be generated by a single tabloid is equal to $$\frac{1}{p^d}{k \choose d}\Bigg(\sum_{\substack {\mu \models p \\ \text{into \textit{d} parts}}}{p \choose \mu _1}{p \choose \mu _2}\cdots {p \choose \mu _d} + \sum_{\substack {\mu \models p \\ \text{into \textit{d}+1 parts}}}{p \choose \mu _1}{p \choose \mu _2}\cdots {p \choose \mu _d} {a_0 \choose \mu _{d+1}}\Bigg). $$\\
\end{enumerate}

\textbf{Case 2:} $r<p$\\
\begin{enumerate}[(i)]
\item The second row of any tabloid $t$ has $r$ entries, $r<p$. Therefore, if any of the $r$ entries belong to $B_i$ for some $i \leq k$, then dim ($<t>$)$\geq p$. So the only tabloids that generate a 1-dimensional submodule should have all the entries choosen from $B_{k+1} \cup ... \cup B_{k+a_0}$. There are ${a_0 \choose r}$ 1-dimensional submodules.\\
\item A tabloid $t$ such that dim ($<t>$)$=p^d$, for any $d\geq 1$ can be obtained in two different ways. First, choose exactly $d$ distinct sets from the collection $\{B_1,B_2,...,B_k\}$. \\
\begin{enumerate}[(a)]
\item Consider $\mu$ a composition of $r$ into exactly $d$ parts and choose $\mu _i$ entries from a chosen $B_i$. Therefore, for a fixed $d$ element subset of the set $\{B_1,B_2,...,B_k\}$, the total number of tabloids of the above kind possible are $$\sum_{\substack {\mu \models r \\ \text{into \textit{d} parts}}}{p \choose \mu _1}{p \choose \mu _2}\cdots {p \choose \mu _d}.$$
Hence, for a fixed choice of $d$ element set we get $$\frac{1}{p^d}\sum_{\substack {\mu \models r \\ \text{into \textit{d} parts}}}{p \choose \mu _1}{p \choose \mu _2}\cdots {p \choose \mu _d} $$ $p^d$-dimensional submodules that are generated by a single tabloid.\\
\item Consider $\mu$ a composition of $r$ into exactly $d+1$ parts and choose $\mu _{d+1}$ entries from $B_{k+1} \cup ... \cup B_{k+a_0}$. Therefore, for a fixed $d$ element subset of the set $\{B_1,B_2,...,B_k\}$, the total number of tabloids of the above kind possible are $$\sum_{\substack {\mu \models r \\ \text{into \textit{d}+1 parts}}}{p \choose \mu _1}{p \choose \mu _2}\cdots {p \choose \mu _d}{a_0 \choose \mu _{d+1}}$$
and we get $$\frac{1}{p^d}\sum_{\substack {\mu \models r \\ \text{into \textit{d}+1 parts}}}{p \choose \mu _1}{p \choose \mu _2}\cdots {p \choose \mu _d} {a_0 \choose \mu _{d+1}} $$ $p^d$-dimensional submodules that are generated by a single tabloid.\\
\end{enumerate}
For our fixed choice of $d$ element subset of $\{B_1,B_2,...,B_k\}$, the submodules obtained are isomorphic. The number of non-isomorphic submodules depends on the choice of the $d$ element subset of $\{B_1,B_2,...,B_k\}$, which can be chosen in ${k \choose d}$ ways. Hence, the total number of $p^d$-dimensional submodules which may be generated by a single tabloid is equal to $$\frac{1}{p^d}{k \choose d}\Bigg(\sum_{\substack {\mu \models r \\ \text{into \textit{d} parts}}}{p \choose \mu _1}{p \choose \mu _2}\cdots {p \choose \mu _d} + \sum_{\substack {\mu \models r \\ \text{into \textit{d}+1 parts}}}{p \choose \mu _1}{p \choose \mu _2}\cdots {p \choose \mu _d} {a_0 \choose \mu _{d+1}}\Bigg). $$\\
\end{enumerate}

\textbf{Case 3:} $r>p$ (let $r=qp+b_0$)\\
\begin{enumerate}[(i)]
\item The second row of any tabloid $t$ has $r$ entries, $r>p$. Therefore, if any of the $r$ entries belong to $B_i$ for some $i \leq k$, then dim ($<t>$)$\geq p$ unless every element of $B_i$ is an entry of the second row of the tabloid. So the only tabloids that generate a 1-dimensional submodule should have all the entries of some $B_i$, $i \leq k$ and remaining entries choosen from $B_{k+1} \cup ... \cup B_{k+a_0}$. There are $${k \choose q}{a_0 \choose b_0}$$ tabloids that generate a 1-dimensional submodule.\\
\item A tabloid $t$ such that dim ($<t>$)$=p^d$, for any $d\geq 1$ can be obtained in two different ways. First, choose exactly $d$ distinct sets from the collection $\{B_1,B_2,...,B_k\}$. \\
\begin{enumerate}[(a)]
\item For $0 \leq j \leq d$, consider $\mu$ a composition of $jp+b_0$ into exactly $d$ parts, such that each part of $\mu$ is less than $p$ and choose $\mu _i$ entries from a chosen $B_i$. Therefore, for a fixed $d$ element subset of the set $\{B_1,B_2,...,B_k\}$, the total number of tabloids of the above kind possible are $$\sum_{j=0}^d{k-d \choose q-j} \sum_{\substack {\mu \models jp+b_0 \\ \text{into \textit{d} parts}\\ \text{each} < \text{p}}} {p \choose \mu _1}{p \choose \mu _2}\cdots {p \choose \mu _d}.$$
Hence, for a fixed choice of $d$ element set we get $$\frac{1}{p^d}
\sum_{j=0}^d {k-d \choose q-j} \sum_{\substack {\mu \models jp+b_0 \\ \text{into \textit{d} parts}\\ \text{each} < \text{p}}}
{p \choose \mu _1}{p \choose \mu _2}\cdots {p \choose \mu _d} $$ $p^d$-dimensional submodules that are generated by a single tabloid.\\
\item For $0 \leq j \leq d$, consider $\mu$ a composition of $jp+b_0$ into exactly $d+1$ parts, such that each part of $\mu$ is less than $p$ and choose $\mu _i$ entries from a chosen $B_i$ and $\mu_{d+1}$ entries from $B_{k+1} \cup ... \cup B_{k+a_0}$. Therefore, for a fixed $d$ element subset of the set $\{B_1,B_2,...,B_k\}$, the total number of tabloids of the above kind possible are $$\sum_{j=0}^d{k-d \choose q-j} \sum_{\substack {\mu \models jp+b_0 \\ \text{into \textit{d}+1 parts}\\ \text{each} < \text{p}}} {p \choose \mu _1}{p \choose \mu _2}\cdots {p \choose \mu _d}{a_0 \choose \mu _{d+1}}.$$
Hence, we get $$\frac{1}{p^d}
\sum_{j=0}^d {k-d \choose q-j} \sum_{\substack {\mu \models jp+b_0 \\ \text{into \textit{d}+1 parts}\\ \text{each} < \text{p}}}
{p \choose \mu _1}{p \choose \mu _2}\cdots {p \choose \mu _d}{a_0 \choose \mu _{d+1}} $$ $p^d$-dimensional submodules that are generated by a single tabloid.\\
\end{enumerate}
Again, the number of non-isomorphic submodules depends on the choice of the $d$ element subset of $\{B_1,B_2,...,B_k\}$, which can be chosen in ${k \choose d}$ ways. Hence, the total number of $p^d$-dimensional submodules which may be generated by a single tabloid is equal to $$\frac{1}{p^d}{k \choose d}\Bigg(\sum_{j=0}^d {k-d \choose q-j} \Bigg( \sum_{\substack {\mu \models jp+b_0 \\ \text{into \textit{d} parts}\\ \text{each} < \text{p}}}
{p \choose \mu _1}{p \choose \mu _2}\cdots {p \choose \mu _d}+$$$$\sum_{\substack {\mu \models jp+b_0 \\ \text{into \textit{d}+1 parts}\\ \text{each} < \text{p}}}
{p \choose \mu _1}{p \choose \mu _2}\cdots {p \choose \mu _d}{a_0 \choose \mu _{d+1}} \Bigg) \Bigg). 
  $$
\\

\end{enumerate}

\begin{rem} \label{summandsofres}
\begin{enumerate}[(i)]
\item The work above gives us a direct sum decomposition of $M_P$, so that $$M_P=\bigoplus_{\alpha \in I} A_\alpha$$
with $\operatorname{dim} A_\alpha = p^{d_\alpha}$ for some non-negative integer $d_\alpha$.\\
\item Each summand is cyclic and has as its basis a single orbit. We will show in Theorem \ref{summandindec} that they are indecomposable.\\
\end{enumerate}
\end{rem}
\textbf{(II) Tensor products of indecomposable summands}\\
\\
Given a tabloid $t$ corresponding to the partition $(n-r,r)$, $2r \leq n$, we can determine the dim $<t>$ by analysing the second row of $t$. Count the number of blocks $B_1,...,B_k$ that are used to fill up the entries of the second row of $t$, call it $N$ and then count the number of blocks that have been entirely used (i.e., if all the $p$ elements of the block are used), call it $D$. Let $d=N-D$, then dim $<t>=p^d$. We can say that the tabloid $t$ is constituted from $d$ of these blocks.
\smallskip

\begin{notn} Let $A_i ^{(d)}$ be a summand of $M_P$ as described in Remark \ref{summandsofres} such that dim $A_i ^{(d)}=p^d$. Let $A_i ^{(d)}=<t_i>$, for some tabloid $t_i$.
\end{notn}
 
\smallskip
\begin{lem} \label{indecomp} If $t$ is constituted from the blocks $B_1,...,B_d$, then the restriction of $<t>$ to $Q$ is isomorphic to $\mathbf{k} Q$ where $Q=<(1,2,...,p),(p+1,...,2p),...,((d-1)p+1,...,dp)>$.
\end{lem}
\begin{proof}
Let $\phi:\mathbf{k} Q \to <t>$ be defined by $g \mapsto g \cdot t$. This map gives an isomorphism between $\mathbf{k}Q$ and the restriction of $<t>$ to $Q$. 
\end{proof}
\smallskip
\begin{cor} \label{iso} Let $A_i ^{(d)}=<t_i>$ and $A_j ^{(d)}=<t_j>$. $A_i ^{(d)} \cong A_j ^{(d)}$ if and only if both $t_i$ and $t_j$ are constituted by the same set of blocks.
\end{cor}
\begin{proof}
Consider $Q$ a subgroup of $P$ generated by $d$ of the $p$-cycles that correspond to the $d$ blocks that constitute $t_i$ and $t_j$. Then, both $<t_i>$ and $<t_j>$ are isomorphic when restricted to $Q$ and the generators of $P$ that are not in $Q$ act trivially on $t_i$ and $t_j$.
\end{proof}
\smallskip
\begin{thm} \label{summandindec}  The cyclic summands of $M_P$ found above are indecomposable.
\end{thm}
\begin{proof}
Let $A_\alpha$ be a summand of $M_P$. We know that $A_\alpha$ is cyclic so let $A_\alpha=<t>$ for some tabloid $t$. Let $Q$ be the subgroup of $P$ generated by the $p$-cycles that constitute $t$, then by Lemma \ref{indecomp} we know that the restriction of $A_\alpha$ to $Q$ is isomorphic to $\mathbf{k} Q$. So $A_\alpha$ is an indecomposable $\mathbf{k} P$-module. 
\end{proof}
\smallskip
\begin{rem} \label{projsummand} If $d=k$ then $A_i ^{(d)}$ is projective. Therefore if $A_i ^{(d)}$ is a non-projective summand then $d<k$.
\end{rem}
\smallskip
\begin{lem} \label{dimen} $A_i ^{(d)} \otimes A_i ^{(d)} = p^d A_i ^{(d)}$.
\end{lem}  
\begin{proof}
Consider $Q$ a subgroup of $P$ of rank $d$ generated by the $p$-cycles that correspond to the $d$ blocks that constitute $t$. Proof follows from Lemma \ref{indecomp} and Corollary \ref{iso}.
\end{proof}
\smallskip
\begin{lem} \label{non-iso} If $A_i ^{(d)} \ncong A_j ^{(d')}$ then the indecomposable summands of $A_i ^{(d)} \otimes A_j ^{(d')}$ are of the form $A_k ^{(e)}$, for some $e > \operatorname{max}\{d,d'\} $.
\end{lem}
\begin{proof}
Let $A_i ^{(d)}=<t_i>$ and $A_j ^{(d')}=<t_j>$ be constituted from the set $\mathcal{B}_i, \mathcal{B}_j \subset \{B_1,...,B_k\}$ respectively with $|\mathcal{B}_i|=d$ and $|\mathcal{B}_j|=d'$ and $\mathcal{B}_i \neq \mathcal{B}_j$. Let $e=|\mathcal{B}_i \cup \mathcal{B}_j|> \operatorname{max}\{d,d'\} $. Let $t_k$ be a tabloid constituted from the $e$ blocks in the set $\mathcal{B}_i \cup \mathcal{B}_j$. Then, $<t_k> \cong <t_i \otimes t_j>$ with the isomorphism $\psi: <t_k> \to <t_i \otimes t_j>$ given by $g \cdot t_k \mapsto g \cdot t_i \otimes g \cdot t_j$.
\end{proof}
\smallskip
\begin{cor} \label{whenproj} Let $\{t_{\alpha}\}_{\alpha \in I}$ be a collection of tabloids constituted from the blocks $\{\mathcal{B}_{\alpha}\}_{\alpha \in I}$, where $\mathcal{B}_{\alpha} \subset \{B_1,...,B_k\}$ for every $\alpha \in I$ such that $$\bigcup_{\alpha \in I} \mathcal{B}_{\alpha}=\{B_1,...,B_k\}.$$ Then, $$<\bigotimes_{\alpha \in I} t_{\alpha}>$$ is a projective $kP$-module.
\end{cor}

\begin{proof}
Proof follows from Lemma \ref{non-iso} and Remark \ref{projsummand}.
\end{proof}
\smallskip
\section{The Benson - Symonds invariant for $M_P$}

In this section we calculate the Benson - Symonds invariant for $M_P$. We know from \cite{KErdmannYoung} that the indecomposable summands of permutation modules are Young modules and as a consequence of Mackey's theorem, tensor product of Young modules decomposes into a direct sum of Young modules. Hence, only finitely many summands occur in tensor powers of $M$ and consequently for $M_P$. We know therefore from Remark \ref{count} that it is enough to count the number of non-projective indecomposable summands of the $core$ of higher tensor powers of $M_P$. Let
$$core_P(M)=\bigoplus_{\alpha \in I} A_\alpha$$
where each $A_\alpha$ is a non-projective indecomposable summand with $\operatorname{dim} A_\alpha = p^{d_\alpha}$ for some non-negative integer $d_\alpha$. Let $\mathcal{B}_\alpha$ be the subset of $\{B_1,...,B_k\}$ containing the blocks that constitute the tabloids in $A_\alpha$ so that $|\mathcal{B}_\alpha|=d_\alpha$. We have
$$core_P(M^{\otimes n}) = \bigoplus_{\nu \models n} \Bigg( {n \choose \nu} core \Big( \bigotimes_\alpha (A_\alpha)^{\otimes \nu _\alpha} \Big )\Bigg) $$
where ${n \choose \nu}$ is the multinomial coefficient and if $\nu_\alpha \geq 1$ then by Lemma \ref{dimen} $$(A_\alpha)^{\otimes \nu _\alpha}= p^{d_\alpha(\nu_\alpha-1)}\cdot A_\alpha$$
Therefore for any $J \subset I$, with $w=|J|$ the coefficient of $\bigotimes_{\alpha\in J} A_\alpha$ is 
\begin{eqnarray}
\sum_{\substack {\nu \models n \\ \text{into w parts}}} {n \choose \nu} p^{\sum_{\alpha \in J} d_\alpha(\nu _\alpha -1)} \label{4.1}
\end{eqnarray}

The coefficient above depends only on $J$. We know from Corollary \ref{whenproj} that $\bigotimes_{\alpha\in J} A_\alpha$ is non-projective if and only if $$\bigcup_{\alpha \in J}\mathcal{B}_\alpha \neq \{B_1,...,B_k\}.$$
In fact the coefficient will be maximum if $J$ is chosen such that $$\bigcup_{\alpha \in J}\mathcal{B}_\alpha =\{B_1,...,B_{k-1}\}$$
making (\ref{4.1}) sum over the largest possible set of compositions of $n$  while being a coefficient of a non-projective summand.\\

\noindent Let $J_0$ be the maximal possible subset of $I$ such that $$\bigcup_{\alpha \in J_0}\mathcal{B}_\alpha =\{B_1,...,B_{k-1}\}.$$

\begin{thm} \label{maxinvt} $$\gamma_P(M)=\sum_{\alpha \in J_0} p^{d_\alpha}$$
\end{thm}
\begin{proof}

Let $w=|J_0|$. Rearrange the indexing set $I$ so that the first $w$ indices in $I$ are those in $J_0$. 
Let
$$N_P=\bigotimes_{\alpha=1}^w A_\alpha.$$
Then, the multiplicity of $N_P$ is maximum. Hence, the Benson - Symonds invariant for $M_P$ is equal to the limit of the $n$th root of the multiplicity of $N_P$ in higher tensor powers of $M_P$. Observe that the coefficient of $N_P$ in $(M_P)^{\otimes n}$ is the same as the coefficient of $N_P$ in $$\Bigg( \bigoplus_{\alpha=1}^w A_\alpha \Bigg)^{\otimes n} $$
Therefore, the limit of the $n$th root of the multiplicity of $N_P$ in $(M_P)^{\otimes n}$ is the same as the limit of the $n$th root of the multiplicity of $N_P$ in $$\Bigg( \bigoplus_{\alpha =1}^w A_\alpha \Bigg)^{\otimes n} $$ which is the same as the limit of the $n$th root of the number of non-projective indecomposable summands in $$\Bigg( \bigoplus_{\alpha=1}^w A_\alpha \Bigg)^{\otimes n} $$
which is 
\begin{eqnarray*}
&=&\sum_{\substack {\nu \models n }} {n \choose \nu} p^{\sum_{\substack{ \alpha \in J_0\\ \nu_\alpha \neq 0} } d_\alpha(\nu _\alpha -1)}\\
&=&\sum_{\substack {\nu \models n }} {n \choose \nu} p^{\sum_{\alpha \in J_0} d_\alpha \nu _\alpha} \cdot p^{-\sum_{\alpha \in J_0} d_\alpha}\\
&=&\frac{1}{p^{\sum_{\alpha \in J_0} d_\alpha}}    \sum_{\substack {\nu \models n }} {n \choose \nu} p^{\sum_{\alpha \in J_0} d_\alpha \nu _\alpha} \\
&=&\frac{1}{p^{\sum_{\alpha \in J_0} d_\alpha}} \bigg (\sum_{\alpha \in J_0} p^{d_\alpha}\bigg )^n\\
\end{eqnarray*}
and 
\begin{eqnarray*}
\gamma_P(M)= \lim _{n \to \infty} \bigg (  \frac{1}{p^{\sum_{\alpha \in J_0} d_\alpha}}  \bigg)^{1/n} \sum_{\alpha \in J_0} p^{d_\alpha} =\sum_{\alpha \in J_0} p^{d_\alpha} 
\end{eqnarray*}
\end{proof}
Recall that $p^{d_\alpha}$ is the dimension of $A_\alpha$ which is constituted by $\mathcal{B}_\alpha \subset \{B_1,...,B_{k-1}\}$. Hence, the Benson - Symonds invariant is really the sum of the dimensions of the indecomposable summands that are constituted by a set of blocks other than the block $B_k$. In the previous section we saw what the indecomposable summands are for the three different cases and what are their multiplicities. We will now evaluate the invariant by adding the dimensions of indecomposable summands that are constituted by blocks other than $B_k$.\\
\\
\textbf{Case 1:} $r=p$, $\lambda = (\lambda_1,\lambda_2)=(n-p,p)$\\
Recall that a $p^d$-dimensional submodule was obtained by choosing a $d$ element subset of $\{B_1,...,B_k\}$. There were a total of ${k \choose d}$ non-isomorphic copies in each case. Now that we want only the submodules constituted by blocks other than $B_k$, we must choose a $d$ element subset of $\{B_1,...,B_{k-1}\}$ and hence there will be ${k-1 \choose d}$ non-isomorphic copies in each case. Multiplying the total number of $p^d$-dimensional submodules with $p^d$ to get the Benson - Symonds invariant we get,
$$\gamma_P(M)= k+ (k-1)\sum_{\substack {\mu \models p \\ \text{into 2 parts}}} {p \choose \mu _1} {a_0 \choose \mu _2}+\sum_{d=2} ^{k-1} {k-1 \choose d}\Bigg(\sum_{\substack {\mu \models p \\ \text{into \textit{d} parts}}}{p \choose \mu _1}{p \choose \mu _2}\cdots {p \choose \mu _d} +$$
$$\sum_{\substack {\mu \models p \\ \text{into \textit{d}+1 parts}}}{p \choose \mu _1}{p \choose \mu _2}\cdots {p \choose \mu _d} {a_0 \choose \mu _{d+1}}\Bigg) $$
$$=1+{n-p \choose p} = {n-p \choose \lambda_1}+{n-p \choose \lambda_2} $$
\textit{See proof in appendix.} Theorem \ref{caser=p} \\
\\
\textbf{Case 2:} $r<p$, $\lambda = (\lambda_1,\lambda_2)=(n-r,r)$\\
$$ \gamma_P(M)={a_0 \choose r}+ \sum_{d=1} ^{k-1}{k-1 \choose d}\Bigg(\sum_{\substack {\mu \models r \\ \text{into \textit{d} parts}}}{p \choose \mu _1}{p \choose \mu _2}\cdots {p \choose \mu _d} +$$ $$\sum_{\substack {\mu \models r \\ \text{into \textit{d}+1 parts}}}{p \choose \mu _1}{p \choose \mu _2}\cdots {p \choose \mu _d} {a_0 \choose \mu _{d+1}}\Bigg) $$
$$={n-p \choose r} = {n-p \choose \lambda_1}+{n-p \choose \lambda_2} $$
\textit{See proof in appendix.} Theorem \ref{caser<p} \\
\\
\textbf{Case 3:} $r>p$, $r=qp+b_0$, $\lambda = (\lambda_1,\lambda_2)=(n-r,r)$\\
$$\gamma_P(M) = {k \choose q}{a_0 \choose b_0}+\sum_{d=1}^{k-1} {k-1 \choose d}\Bigg(\sum_{j=0}^d {k-d \choose q-j} \Bigg( \sum_{\substack {\mu \models jp+b_0 \\ \text{into \textit{d} parts}\\ \text{each} < \text{p}}}
{p \choose \mu _1}{p \choose \mu _2}\cdots {p \choose \mu _d}+$$$$\sum_{\substack {\mu \models jp+b_0 \\ \text{into \textit{d}+1 parts}\\ \text{each} < \text{p}}}
{p \choose \mu _1}{p \choose \mu _2}\cdots {p \choose \mu _d}{a_0 \choose \mu _{d+1}} \Bigg) \Bigg) 
  $$
  $$={n-p \choose r}+{n-p \choose r-p} = {n-p \choose \lambda_1}+{n-p \choose \lambda_2} $$
\textit{See proof in appendix.} Theorem \ref{caser>p} \\
\\
Hence we conclude that 
\begin{eqnarray}
\gamma_P(M)&=&{n-p \choose \lambda_1}+{n-p \choose \lambda_2}. \label{4.3}
\end{eqnarray}
\smallskip

\section{The Benson - Symonds invariant for $M$}

In this section we first look at the invariant for $M$ restricted to an arbitrary maximal elementary abelian $p$-subgroup of $\mathcal{S}_n$ and then use Theorem \ref{maxElt} to determine the invariant for $M$. Let $E$ be an arbitrary maximal elementary abelian $p$-subgroup of $\mathcal{S}_n$.  It is clear that rank $E \leq k$. Using Theorem \ref{conjEltAb}, consider the partition of $\{1,2,...,n\}$ into orbits of $E$ and denote it by $C_1,...,C_m,C_{m+1},...,C_{m+a_0}$ where $|C_i|=p^{c_i}$ for some $c_i \in \mathbb{Z}_{>0}$ for $1 \leq i \leq m$ and $|C_i|=1$ for $i>m$.\\ 
$$core_E(M)=\bigoplus_{\alpha \in I} A_\alpha$$
where each $A_\alpha$ is a cyclic non-projective indecomposable summand generated by a single tabloid with $\operatorname{dim} A_\alpha = p^{d_\alpha}$ for some non-negative integer $d_\alpha$ (same as the size of $E$-orbit of the tabloid). Let $\mathcal{C}_\alpha$ be the subset of $\{C_1,...,C_m\}$ containing the blocks that constitute the tabloids in $A_\alpha$.\\

Choose $1 \leq i_0 \leq m$ such that $c_{i_0}$=min $\{c_i | 1 \leq i \leq m\}$. Let $J$ be the maximal possible subset of $I$ such that $$\bigcup_{\alpha \in J}\mathcal{C}_\alpha =\{C_1,...,C_m\} \setminus \{C_{i_0}\}$$
Then as in Theorem \ref{maxinvt}, we have 
$$\gamma_E(M)=\sum_{\alpha \in J} p^{d_\alpha}$$
So $\gamma_E(M)$ is equal to the number of tabloids in $M$ that are constituted by the blocks other than the block $C_{i_0}$. This can at most be the number of tabloids in $M$ constituted by the blocks other than $B_k$ in the previous section because $|B_k| \leq |C_{i_0}|$.\\
Therefore, $$\gamma_E(M) \leq \gamma_P(M)$$
Hence, we have proved the following theorem:\\
\begin{thm} Let $M^{(n-r,r)}$ be the permutation module over the symmetric group $\mathcal{S}_n$ and let $P$ be an elementary abelian $p$-subgroup of $\mathcal{S}_n$ of highest possible rank. Then,
$$\gamma_{\mathcal{S}_n}(M)=\gamma_P(M)$$
\end{thm}

This completes the proof of Theorem \ref{mytheorem} using (\ref{4.3}) from the previous section and Theorem \ref{maxElt}.
\\

\section{Further Directions}

In this section, we include some observations and further questions. The following conjecture is motivated by the behavior of permutation modules when restricted to elementary abelian $p$-subgroups and also by several computations performed using the computer algebra system \textbf{Magma} \cite{magma}. \\

\begin{conj} If $\lambda \vdash n$ and $M^\lambda$ is the permutation module of the symmetric group $\mathcal{S}_n$, then $\gamma_{\mathcal{S}_n}(M^\lambda)$ is an integer.
\end{conj}

In fact after a number of computations of the invariant for trivial source modules over a finite group, we could not find an example of one with non-integer invariant. This suggests the following question:\\

\begin{quest} If $M$ is a trivial source module of a finite group $G$, then is $\gamma_G(M)$ always an integer?
\end{quest}

The obvious next class of modules to consider would be of Young modules which are indecomposable summands of permutation modules \cite{KErdmannYoung}. We know from \cite{B_S} that the Benson - Symonds invariant does not behave well with direct sums, in that the invariant for the direct sum is sandwiched between the maximum of the individual invariants and the sum of the individual invariants. If we can find a recursive formula using the decomposition of permutation modules to find the invariant for the Young modules, it would be a delight.\\

In particular, we know that corresponding to the partition $\lambda=(n-1,1)$,  
$$M^\lambda=\left\{\begin{array}{lr}
        Y^\lambda, &  p\mid n\\
        \mathbf{k} \oplus Y^\lambda , & p \nmid n\\
        \end{array} \right. $$ 
        
Hence, $$\gamma_{\mathcal{S}_n}(Y^\lambda)   =\left\{\begin{array}{lr}
        n-p, &  p\mid n\\
        n-p-1 , & p \nmid n\\
        \end{array} \right.$$   
        
So, we have the following result:\\
\begin{thm} If $\lambda=(n-1,1)$ then $\gamma_{\mathcal{S}_n}(Y^\lambda)=\operatorname{dim} Y^\lambda -p$.
\end{thm}

The next class of modules which is extensively studied in the representation theory of symmetric groups is of Specht modules $S^\lambda$. The main obstruction to our method in the case of Specht modules is that the indecomposable summands of tensor products of Specht modules is not known. If we have the knowledge of the number of composition factors of higher tensor powers of $S^\lambda$ or even the number of composition factors of the socle of higher tensor powers of $S^\lambda$, there is some hope to work out the invariant.\\

\newpage

\appendix

\section{}

\begin{lem} (Chu-Vandermonde identity) \label{chuvand}
$$\sum_i {r \choose i}{s \choose n-i}={r+s \choose n}$$
\end{lem}

This can be generalized to the following:
\begin{lem} 
$$\sum_{\substack {\mu \models r }}{p \choose \mu _1}{p \choose \mu _2}\cdots {p \choose \mu _d}={dp \choose r}$$
\end{lem}

If we fix the number of parts of the composition then using the inclusion - exclusion principle, we have:
\begin{lem} \label{intodparts}
$$\sum_{\substack {\mu \models r \\ \text{into \textit{d} parts}}}{p \choose \mu _1}{p \choose \mu _2}\cdots {p \choose \mu _d} = \sum_{i=0}^{d} (-1)^i{d \choose i} {(d-i)p \choose r}$$
\end{lem}

A more general form of the above Lemma would be the following: 
\begin{lem} \label{intod+1parts}
$$\sum_{\substack {\mu \models r \\ \text{into \textit{d}+1 parts}}}{p \choose \mu _1}{p \choose \mu _2}\cdots {p \choose \mu _d}{a_0 \choose \mu_{d+1}} = \sum_{i=0}^{d} (-1)^i{d \choose i} \Bigg({(d-i)p+a_0 \choose r}- {(d-i)p \choose r}\Bigg) $$
\end{lem}

Adding Lemma \ref{intodparts} and Lemma \ref{intod+1parts} we get:
\begin{lem} \label{addingdandd+1}
$$\sum_{\substack {\mu \models r \\ \text{into \textit{d} parts}}}{p \choose \mu _1}{p \choose \mu _2}\cdots {p \choose \mu _d}+\sum_{\substack {\mu \models r \\ \text{into \textit{d}+1 parts}}}{p \choose \mu _1}{p \choose \mu _2}\cdots {p \choose \mu _d}{a_0 \choose \mu_{d+1}}$$ $$ = \sum_{i=0}^{d} (-1)^i{d \choose i} {(d-i)p+a_0 \choose r} $$
\end{lem}

Note that in the above Lemma, the compositions with any part $>p$ do not contribute to the sum. So if we were to enforce each part of the composition to be less than $p$ then we only need to avoid the partitions with some part equal to $p$. Using inclusion-exclusion principle, we have:
\smallskip
\begin{lem} \label{simp3}
$$ \sum_{\substack {\mu \models jp+b_0 \\ \text{into \textit{d} parts}\\ \text{each} < \text{p}}}
{p \choose \mu _1}{p \choose \mu _2}\cdots {p \choose \mu _d}+\sum_{\substack {\mu \models jp+b_0 \\ \text{into \textit{d}+1 parts}\\ \text{each} < \text{p}}}
{p \choose \mu _1}{p \choose \mu _2}\cdots {p \choose \mu _d}{a_0 \choose \mu _{d+1}} $$ $$= \sum_{h=0}^{j} (-1)^h{d \choose h} \sum_{i=0}^{d-h} (-1)^i {d-h \choose i} {(d-h-i)p+a_0 \choose (j-h)p+b_0}   $$
\end{lem}

\smallskip

\begin{lem} \label{simp1}
$$\sum_{i=0}^{n} (-1)^i {n \choose i}{n-i \choose m}=\left\{\begin{array}{lr}
        1, &  m= n\\
        0, & \text{otherwise}
        \end{array} \right. $$ 
\end{lem}
\begin{proof}
\begin{eqnarray*}
\sum_{i=0}^{n} (-1)^i {n \choose i}{n-i \choose m}&=&\sum_{i=0}^{n} (-1)^i {n \choose m}{n-m \choose i}\\
&=&{n \choose m}\sum_{i=0}^{n} (-1)^i {n-m \choose i}\\
&=& \left\{\begin{array}{lr}
        1, &  m= n\\
        0, & \text{otherwise}
        \end{array} \right.\\
\end{eqnarray*}
\end{proof}

\smallskip
\begin{lem} \label{coeff3}
$$(-1)^m \sum_{i=m}^{k-1} (-1)^i {k-1 \choose i} {i \choose m} \sum_{j=0}^{i-m} {i-m \choose j} {k-i \choose r-j}  
$$
$$=\left\{\begin{array}{lr}
        1, &  m=k-1,  r=0,1\\
        0, & \text{otherwise}
        \end{array} \right. $$ 
\end{lem}

\smallskip
\begin{proof}
$$(-1)^m \sum_{i=m}^{k-1} (-1)^i {k-1 \choose i} {i \choose m} \sum_{j=0}^{i-m} {i-m \choose j} {k-i \choose r-j}$$
\begin{eqnarray*}
&=& (-1)^m \sum_{i=m}^{k-1} (-1)^i {k-1 \choose i} {i \choose m} {k-m \choose r}  \text{     Using Lemma } \ref{chuvand}   \\
&=& (-1)^m {k-m \choose r} \sum_{i=m}^{k-1} (-1)^i {k-1 \choose i} {i \choose m} \\
&=& (-1)^m {k-m \choose r} \sum_{i=m}^{k-1} (-1)^i {k-1 \choose m} {k-1-m \choose i-m} \\
&=& (-1)^m {k-m \choose r} {k-1 \choose m} \sum_{i=m}^{k-1} (-1)^i  {k-1-m \choose i-m} \\
&=& (-1)^m {k-m \choose r} {k-1 \choose m} \sum_{i=0}^{k-1-m} (-1)^{m+i}  {k-1-m \choose i} \\
&=&\left\{\begin{array}{lr}
        1, &  m=k-1,  r=0,1\\
        0, & \text{otherwise}
        \end{array} \right.
\end{eqnarray*}
\end{proof}

\begin{thm} \label{caser=p}
$$k+ (k-1)\sum_{\substack {\mu \models p \\ \text{into 2 parts}}} {p \choose \mu _1} {a_0 \choose \mu _2}+\sum_{d=2} ^{k-1} {k-1 \choose d}\Bigg(\sum_{\substack {\mu \models p \\ \text{into \textit{d} parts}}}{p \choose \mu _1}{p \choose \mu _2}\cdots {p \choose \mu _d} +$$
$$\sum_{\substack {\mu \models p \\ \text{into \textit{d}+1 parts}}}{p \choose \mu _1}{p \choose \mu _2}\cdots {p \choose \mu _d} {a_0 \choose \mu _{d+1}}\Bigg) $$
$$=1+{n-p \choose p} $$
\end{thm}
\begin{proof}
Simplifying the sum on the left hand side using Lemma \ref{chuvand} for the first part and Lemma \ref{intodparts} and Lemma \ref{intod+1parts} for the second part, we have
$$k+ (k-1)\Big( {p+a_0 \choose p}-1 \Big ) +\sum_{d=2} ^{k-1} {k-1 \choose d}\Bigg( \sum_{i=0}^{d-1} (-1)^i{d \choose i} {(d-i)p \choose p}  +$$
$$ \sum_{i=0}^{d-1} (-1)^i{d \choose i} \Bigg({(d-i)p+a_0 \choose p}- {(d-i)p \choose p}\Bigg)  \Bigg) $$
Now simplifying the first part and using Lemma \ref{addingdandd+1} for the second part we get\\
\begin{eqnarray}
&=&1+ (k-1){p+a_0 \choose p} +\sum_{d=2} ^{k-1} {k-1 \choose d} \sum_{i=0}^{d-1} (-1)^i{d \choose i} {(d-i)p+a_0 \choose p} \notag \\
&=&1+ \sum_{d=1} ^{k-1} {k-1 \choose d} \sum_{i=0}^{d-1} (-1)^i{d \choose i} {(d-i)p+a_0 \choose p} \label{A.10}    
\end{eqnarray}
\\
The coefficient of $${mp+a_0 \choose p}$$ in the above expansion is
$$\sum_{i=0}^{k-1} (-1)^i{k-1 \choose m+i}{m+i \choose i}$$
$$=\sum_{i=0}^{k-1} (-1)^i{k-1 \choose i}{k-1-i \choose m}$$
which is equal to 1 when $m=k-1$ and $0$ otherwise by Lemma \ref{simp1}.
So (\ref{A.10}) reduces to $$1+ {(k-1)p+a_0 \choose p}$$
which is $$1+{n-p \choose p}$$
\end{proof}

\begin{thm} \label{caser<p}
$$ {a_0 \choose r}+ \sum_{d=1} ^{k-1}{k-1 \choose d}\Bigg(\sum_{\substack {\mu \models r \\ \text{into \textit{d} parts}}}{p \choose \mu _1}{p \choose \mu _2}\cdots {p \choose \mu _d} +$$ $$\sum_{\substack {\mu \models r \\ \text{into \textit{d}+1 parts}}}{p \choose \mu _1}{p \choose \mu _2}\cdots {p \choose \mu _d} {a_0 \choose \mu _{d+1}}\Bigg) $$
$$={n-p \choose r} $$
\end{thm}
\begin{proof}
Simplifying the sum on the left hand side using Lemma \ref{intodparts} and Lemma \ref{intod+1parts}, we have
$$ {a_0 \choose r}+ \sum_{d=1} ^{k-1}{k-1 \choose d}\Bigg( \sum_{i=0}^{d} (-1)^i{d \choose i} {(d-i)p \choose r}+$$ $$ \sum_{i=0}^{d} (-1)^i{d \choose i} \Bigg({(d-i)p+a_0 \choose r}- {(d-i)p \choose r}\Bigg) $$
\begin{eqnarray}
&=&{a_0 \choose r}+ \sum_{d=1} ^{k-1}{k-1 \choose d} \sum_{i=0}^{d} (-1)^i{d \choose i}{(d-i)p+a_0 \choose r}\notag \\
&=&{a_0 \choose r}+ \sum_{d=1} ^{k-1}{k-1 \choose d}\Bigg (  \sum_{i=0}^{d-1} (-1)^i{d \choose i}{(d-i)p+a_0 \choose r}+ (-1)^d{a_0 \choose r} \Bigg) \notag \\
&=&{a_0 \choose r}+ \sum_{d=1} ^{k-1}{k-1 \choose d} \sum_{i=0}^{d-1} (-1)^i{d \choose i}{(d-i)p+a_0 \choose r}+ \sum_{d=1} ^{k-1}{k-1 \choose d}   (-1)^d{a_0 \choose r}  \label{A.12}   
\end{eqnarray}
The coefficient of $${mp+a_0 \choose r}$$ in the expansion of the sum in the center is
$$\sum_{i=0}^{k-1} (-1)^i{k-1 \choose m+i}{m+i \choose i}$$
$$=\sum_{i=0}^{k-1} (-1)^i{k-1 \choose i}{k-1-i \choose m}$$
which is equal to 1 when $m=k-1$ and $0$ otherwise by Lemma \ref{simp1}. So, (\ref{A.12}) simplifies to \\
\begin{eqnarray*}
&=&{a_0 \choose r}+{(k-1)p+a_0 \choose r}-{a_0 \choose r}\\
&=&{(k-1)p+a_0 \choose r}\\
&=&{n-p \choose r}\\
\end{eqnarray*}
\end{proof}

\smallskip

\begin{thm} \label{caser>p}
$${k \choose q}{a_0 \choose b_0}+\sum_{d=1}^{k-1} {k-1 \choose d}\Bigg(\sum_{j=0}^d {k-d \choose q-j} \Bigg( \sum_{\substack {\mu \models jp+b_0 \\ \text{into \textit{d} parts}\\ \text{each} < \text{p}}}
{p \choose \mu _1}{p \choose \mu _2}\cdots {p \choose \mu _d}+$$$$\sum_{\substack {\mu \models jp+b_0 \\ \text{into \textit{d}+1 parts}\\ \text{each} < \text{p}}}
{p \choose \mu _1}{p \choose \mu _2}\cdots {p \choose \mu _d}{a_0 \choose \mu _{d+1}} \Bigg) \Bigg) 
  $$
  $$={(k-1)p+a_0 \choose qp+b_0}+{(k-1)p+a_0 \choose (q-1)p+b_0}$$
\end{thm}
\begin{proof}
Using Lemma \ref{simp3} we have,
$$\sum_{d=0}^{k-1} {k-1 \choose d}\Bigg(\sum_{j=0}^d {k-d \choose q-j} \Bigg( \sum_{\substack {\mu \models jp+b_0 \\ \text{into \textit{d} parts}\\ \text{each} < \text{p}}}
{p \choose \mu _1}{p \choose \mu _2}\cdots {p \choose \mu _d}+$$$$\sum_{\substack {\mu \models jp+b_0 \\ \text{into \textit{d}+1 parts}\\ \text{each} < \text{p}}}
{p \choose \mu _1}{p \choose \mu _2}\cdots {p \choose \mu _d}{a_0 \choose \mu _{d+1}} \Bigg) \Bigg) 
  $$
$$=\sum_{d=0}^{k-1} {k-1 \choose d} \sum_{j=0}^d {k-d \choose q-j}
\sum_{h=0}^{j} (-1)^h{d \choose h} \sum_{i=0}^{d-h} (-1)^i {d-h \choose i} {(d-h-i)p+a_0 \choose (j-h)p+b_0} $$ 
For any $m$ and $s$, the coefficient of $${mp+a_0 \choose sp+b_0}$$ is 
$$(-1)^m \sum_{i=m}^{k-1} (-1)^i {k-1 \choose i} {i \choose m} \sum_{j=0}^{i-m} {i-m \choose j} {k-i \choose q-s-j}  $$
Using Lemma \ref{coeff3}, we can simplify our left hand side expression to $$={(k-1)p+a_0 \choose qp+b_0}+{(k-1)p+a_0 \choose (q-1)p+b_0}$$

\end{proof}

 \textbf{Acknowledgements:} I sincerely acknowledge the constant help and guidance of my advisor Prof. David J. Hemmer. I would like to thank him for a careful reading of this article and his valuable suggestions.\\
 \\
 
\bibliographystyle{plain}
\bibliography{bibliography}

\end{document}